\documentclass[10pt]{article}

\usepackage{latexsym, amssymb, amsmath, amsthm, amscd}
\usepackage[all,cmtip]{xy}
\usepackage{amsmath}
\usepackage{amsthm}
\usepackage{amssymb}
\usepackage{mathtools}
\usepackage{amsfonts}
\usepackage{amsrefs}
\usepackage{color,authblk}
\usepackage{color,soul}
\usepackage{tikz}
\usepackage{amscd} 
\usepackage{comment}
\usepackage{mathrsfs}
\usepackage[all]{xy}
\usepackage{graphicx}
\usepackage{authblk}
\usepackage{hyperref}
\usepackage{fullpage}
\usepackage[all,cmtip]{xy}
\usepackage{amsmath,amscd}
\usepackage{tikz-cd}
\usepackage{tikz}
\usetikzlibrary{matrix}

\newcommand{\set}[2]{\left\{ #1 \,\middle|\, #2 \right\}}

\usepackage{hyperref}

\numberwithin{equation}{subsection} \DeclareMathSizes{2}{10}{12}{13}
\theoremstyle{plain}
\newtheorem{thm}{Proposition}[section]

\newtheorem{cor}[thm]{Corollary}
\newtheorem{lem}[thm]{Lemma}
\newtheorem{Thm}[thm]{Theorem}
\theoremstyle{remark}
\newtheorem{rem}[thm]{Remark}
\theoremstyle{definition}
\newtheorem{defn}[thm]{Definition}
\newtheorem{notn}[thm]{Notation}
\newtheorem{cnst}[thm]{Construction}

\numberwithin{thm}{section} 

\title{Noncommutative tensor triangulated categories and coherent frames}

\author{Vivek Mohan Mallick\footnote{Department of Mathematics, Indian
    Insititute of Science Education and Research (IISER) Pune, Pune 411008,
    India,\\
  email: vmallick@iiserpune.ac.in} \hspace{1in} Samarpita Ray \footnote{Center for Geometry and Physics, Institute for Basic Science (IBS), Pohang 37673, South Korea\\
email: ray.samarpita31@gmail.com}   }

\date{}

\begin{document}
	
	\maketitle 
	
	\medskip

	\begin{abstract} 
	  We develop a point-free approach for constructing
	  the Nakano-Vashaw-Yakimov-Balmer spectrum of a noncommutative
	  tensor triangulated category under some mild assumptions. In
	  particular, we provide a conceptual way of classifying radical
	  thick tensor ideals of a noncommutative
	  tensor triangulated category using frame theoretic methods,
	  recovering the universal support data in the process. We further
	  show that there is a homeomorphism between the spectral space of
	  radical thick tensor ideals of a noncommutative
	  tensor triangulated category and the collection of open subsets of its spectrum in the Hochster dual topology.
	\end{abstract}
	
	\medskip
	MSC(2020) Subject Classification:   06D22, 18F70, 18G80, 18M05, 54Exx, 55P43


	\section{Introduction}
	
	The subject of tensor triangular geometry has been an active area of
	research for the past two decades and has touched a wide range of
	areas in mathematics including algebraic geometry, modular
	representation theory, stable homotopy theory, noncommutative
	topology, to name a few. The subject involves the study of
	triangulated categories with a given biexact symmetric monoidal
	functor called the tensor. Balmer \cite{balmer:preshvtrcat,
	balmer:spectrcat} showed that the triangulated category of perfect
	complexes over a scheme $X$ along with the derived tensor functor
	contains enough data to reconstruct $X$, establishing that the
	subject is rich enough to be studied. Associated to a tensor
	triangulated category $\mathcal{C}$, Balmer \cite{balmer:spectrcat}
	constructed a locally ringed space $Spec\ \mathcal{C}$ called
	the spectrum of $\mathcal{C}$ which carries the geometric essence of
	the tensor triangulated category. For example,
	$Spec\bigl(\mathrm{Dperf}(X)\bigr) \cong X$ for a quasi-compact,
	quasi-separated scheme $X$ (see \cite{balmer:spectrcat, MR2280286}).
	The construction of the spectrum involves constructing a space out
	of prime thick tensor ideals of the tensor triangulated category.
	Balmer showed that this satisfies the correct universal property
	among all spaces which act as targets for support data.  In the case
	of modular representation theory, where the relevant tensor
	triangulated category is the stable module category of modules over
	a finite group scheme $G$, the spectrum recovers the projective
	support variety.

	In addition to capturing the underlying scheme, the subject of tensor
	triangular geometry further lifts to this abstraction, the
	notions of finite {\'e}tale maps
	\cite{balmerdellivo:restrictedetale}, the Chow group and
	intersection theory \cite{balmer:ttchowgp,MR3448183,MR3423452},
	Grothendieck-Neeman duality, Wirthm{\"u}ller isomorphism
	\cite{balmerdellivo:gnduality} among others. The theory also detects
	(the failure of) Gersten conjecture for singular schemes
	\cite{MR2439430}. This demonstrates the richness of the theory.

	However, all tensor structures on triangulated categories need not
	be symmetric. The basic examples being stable module categories over
	Hopf algebras. To study these Nakano, Vashaw and Yakimov
	\cite{1909.04304} introduced a noncommutative version of tensor
	triangulated categories and, extending Balmer's theory, constructed
	a topological space in terms of prime thick tensor ideals. This
	construction is also universal in a manner parallel to Balmer's
	construction. Nakano et al prove that in this case again, the space
	corresponds to the projective support variety described in terms of
	the cohomology of the Hopf algebra.

	Following the success of the theories a natural direction of
	exploration will be to understand the construction of spectrum
	itself to gain a better insight into the structure of the tensor
	triangulated category. A conceptual and formal way of constructing
	the spectrum was described by Kock and Pitsch \cite{kopi:hochdual}
	using the language of frames and locales (point free topology).
	Frames are complete lattices where finite meet distributes over
	arbitrary joins (see \cite{johnstone:stonespaces} or section \ref{sec:frames}).  A
	typical example of a frame is the lattice of open subsets of a
	topological space. The essence of point free approach to toplogy is
	to reduce the study of topology to the study of these frames.  In
	this approach, one constructs the topological space in terms of the
	frame of open sets instead of starting with a set of points.
	Spectral spaces are topological spaces homeomorphic to the spectrum
	of a commutative ring. The spectral spaces correspond to coherent
	frames (see Definition \ref{def:cohframe}) in the point free
	approach. Kock and Pitsch gave a point free description of the
	Balmer spectrum of a (commutative) tensor triangulated category
	$\mathcal{C}$ as the Hochster dual of the spectral space associated
	to the \emph{coherent} frame of radical thick tensor ideals of the
	tensor triangulated category.  They also define a notion of support
	taking values in a frame and prove that the coherent frame of
	radical thick tensor ideals is universal as a target for such
	supports. Kock and Pitsch's paper shows that one can arrive at
	various results about the Balmer spectrum (including the sheaf of
	rings)  from this viewpoint.

	In this paper, we explore the noncommutative Balmer spectrum
	studied by Nakano et al from a frame
	theoretic viewpoint. However, in this case the arguments have to be
	modified substantially due to the lack of commutativity of tensor
	and consequently the failure of the construction of the radical of
	an ideal by adjoining $n$-th roots (where $n$ is a natural number).
	For the modified arguments to work, one needs to restrict to a class
	of noncommuative tensor triangulated categories, where the prime
	ideals (defined in terms of thick tensor ideals as $IJ \subseteq P
	\implies (I \subseteq P) \text{ or } (J \subseteq P)$) are complete primes
	(defined in terms of objects, i.e., $A \otimes B \in P \implies (A
	\in P) \text{ or } (B \in P)$). As Nakano et al \cite{nvy:supptpp} shows there is
	a rich class of examples where this holds, for instance the stable module
	category of any finite dimensional Hopf algebra. Following Kock and
	Pitsch, we extend the notion of a frame theoretic support data to
	the noncommuative setup and prove the relevant universal properties
	to recover the spectrum.
	Let $\mathbf{K}$ be a (noncommuative)
	tensor triangulated category.
	The radical thick tensor ideals form a coherent frame (Theorem
	\ref{cohfr}) and the association of $a \in Ob(\mathbf{K})$ to the
	smallest radical thick tensor ideal containing $a$ gives a universal
	frame-theoretic support datum (Theorem \ref{thm:initsupp}) giving us
	a classification of radical thick tensor ideals (Theorem
	\ref{corres}).
	The relation of this construction with the
	Nakano-Vashaw-Yakimov-Balmer spectrum is clarified in Corollary
	\ref{Hdual}.
	Nakano et al's construction of the universal support data taking
	values in the Balmer spectrum is recovered in a frame theoretic way
	in Proposition \ref{pro:cmpkpnvy}.
	Finally, extending a result by Banerjee \cite{MR3787524}, we also show
	that there is a homeomorphism between the set of radical thick
	tensor ideals and the set of closed subsets of the spectrum with
	quasi-compact complements under the proper notions of topologies on
	these sets (see Theorem \ref{noncomTN}).

	\vspace{.1in}

	\noindent\textbf{Acknowldegement}: The first author would like to
	express his gratitude to Indian Institute of Science Education and
	Research (IISER), Pune from where he contributed to this work. The second author's work was partially supported by the SERB NPDF grant PDF/2020/000670 and partially by the grant IBS-R003-D1 of the IBS-CGP, POSTECH, South Korea. The second author would like to sincerely thank both funders for the support.

%	\pagebreak
	\section{Preliminaries}
	\subsection{Noncommutative tensor triangulated category and support}
	A general noncommutative theory of tensor triangular geometry was introduced by Nakano, Vashaw and Yakimov in \cite{1909.04304}. They further studied support maps and its connection with tensor product in the setup of noncommutative tensor triangular geometry in their next paper \cite{nvy:supptpp}. In this section, we recall some definitions and results from \cite{1909.04304} and \cite{nvy:supptpp}.\\
	A noncommutative tensor triangulated category, as introduced in  \cite{1909.04304}, is a triangulated category $\mathbf{K}$ with a biexact monoidal structure. Throughout this paper $\mathbf{K}$ will denote an essentially small  noncommutative tensor triangulated category.
	\begin{defn}[\cite{1909.04304}, $\S 1.2$]
		\begin{itemize}
			\item[(1)] A {\it thick tensor} ideal of $\mathbf{K}$ is a full triangulated subcategory $\mathbf{I}$ of $\mathbf{K}$ such that it contains all direct summands of its objects and for any $A\in Ob(\mathbf{I})$, we have $A\otimes B, B\otimes A\in Ob(\mathbf{I})$ for all $B\in Ob(\mathbf{K})$. 
			\item[(2)] A prime ideal of $\mathbf{K}$ is a proper thick tensor ideal $\mathbf{P}$ such that for all thick tensor ideals $\mathbf{I}$ and $\mathbf{J}$ of $\mathbf{K}$, we have $\mathbf{I}\otimes \mathbf{J} \subseteq \mathbf{P} \implies \mathbf{I}\subseteq \mathbf{P} $ or $\mathbf{J}\subseteq \mathbf{P} $. We denote by $Spc(\mathbf{K})$ the collection of all prime ideals of $\mathbf{K}$.
			\item[(3)] A completely prime ideal of $\mathbf{K}$ is a proper thick tensor ideal $\mathbf{P}$ such that $A\otimes B\in \mathbf{P}\implies A\in \mathbf{P}$ or $B\in \mathbf{P}$ for all $A,B\in Ob(\mathbf{K})$.
		\end{itemize}
	\end{defn}
\begin{defn}[\cite{1909.04304}, $\S 1.2$]\label{SpcK}
	The noncommutative Balmer spectrum $Spc(\mathbf{K})$ of $\mathbf{K}$  is the topological space of prime
	ideals of $\mathbf{K}$  endowed with the Zariski topology which is given by closed sets of the form
	$$V(S)= \{\mathbf{P}\in Spc(\mathbf{K})~|~\mathbf{P}\cap S=\emptyset\}$$
	for all subsets $S$ of $\mathbf{K}$.
\end{defn}
	Let $X$ be a topological space and let $\mathcal{X}_{cl}(X)$ denote the collection of all closed subsets of $X$.
	\begin{defn}[\cite{1909.04304}, Definition 4.1.1] \label{def:nvysuppdata}
		A map $\sigma:\mathbf{K}\longrightarrow \mathcal{X}_{cl}(X)$ is called a (noncommutative) support datum if the following conditions are satisfied:
		\begin{itemize}
			\item[(1)] $\sigma(0)=\emptyset$ and $\sigma(1)=X$
			\item[(2)] $\sigma(A\oplus B)=\sigma(A)\cup \sigma(B), \quad \forall A, B\in Ob(\mathbf{K})$
			\item[(3)] $\sigma(\sum A)= \sigma(A), \quad \forall A\in Ob(\mathbf{K})$
		\item[(4)] If $A\longrightarrow B \longrightarrow C\longrightarrow \sum A$ is a distinguished triangle, then $\sigma(A)\subseteq \sigma(B)\cup \sigma(C)$
		\item[(5)] $\bigcup_{C\in Ob(\mathbf{K)}}\sigma(A\otimes B \otimes C)=\sigma(A)\cap \sigma(B), \quad \forall A, B\in Ob(\mathbf{K})$
	\end{itemize}
	\end{defn}
We recall (see, \cite{1909.04304}*{Lemma 4.1.2}) that the restriction of the map $V$ (as in Definition \ref{SpcK}) to the objects of $\mathbf{K}$ gives a support datum $\mathbf{K}\longrightarrow \mathcal{X}_{cl}(Spc(\mathbf{K}))$	
\begin{Thm}[\cite{1909.04304}, Theorem 4.2.2]
	The support $V$ is final among all the support data $\sigma$ of $\mathbf{K}$ such that $\sigma(A)$ is closed for each $A\in Ob(\mathbf{K})$. Equivalently, for any support datum $\sigma$ satisfying the above condition, there is a unique continuous map $f_\sigma: X \longrightarrow Spc(\mathbf{K})$ such that  $\sigma(A)= f_\sigma^{-1}(V(A))$. This map is precisely given by 
	\begin{equation*}
		f_\sigma(x) = \{A\in Ob(\mathbf{K})~|~ x\notin \sigma(A)\}.
	\end{equation*}
	
\end{Thm}
	\subsection{Coherent frames and support}
	\label{sec:frames}
	
	In this section, we recall some definition and results from \cite{johnstone:stonespaces}, \cite{joyal:specsp} and \cite{kopi:hochdual}.
	\begin{defn}%[\cite{kopi:hochdual}, Definition 1.2.1]
	    A \emph{frame} is a complete lattice which satisfies the infinite distributive law:
		$$a\wedge \underset{s\in S} {\vee}s = \underset{s\in S}{\vee}(a\wedge s).$$
		A frame map is a lattice map that preserves arbitrary joins. The category of frames and frame maps is denoted by $\mathbf{Frm}$.
	\end{defn}
%The opposite category of frames is called the category of locales and is denoted by $\mathbf{Loc}$. 
There is a pair of adjoint functors between the category of topological
spaces $\mathbf{Top}$ and the opposite category of frames
$\mathbf{Frm}^{op}$ which we now recall (\cite{johnstone:stonespaces}*{\S II.1.4}). The open sets of any topological space form a frame with join operation given by union of open sets and finite meet given by intersection. This gives a functor 
$\mathbf{Top}\longrightarrow \mathbf{Frm}^{op}$ which has a right adjoint, the functor of {\it points}. A {\it point} of a frame is a frame map $x:F\longrightarrow \{0,1\}$ where $\{0, 1\}$ is the
Boolean algebra of two elements (with $0 < 1$). The set of points of any frame form a topological space whose open sets are given by sets of the form $\{x: F\longrightarrow\{0,1\}~|~x(u)=1\}$ for any $u\in F$ and this gives the functor $\mathbf{Frm}^{op}\longrightarrow \mathbf{Top}$.\\
We recall from \cite{johnstone:stonespaces}*{\S II.3.1} that an element $a$ of a frame $F$ is called {\it finite} if for every subset $S\subseteq F$ with $a\leq \underset{s\in S}{\vee} s$, there exists a finite subset $S'\subseteq S$ with $a\leq \underset{s\in S'}{\vee} s$.
\begin{defn}[\cite{johnstone:stonespaces}, \S II.3.2] \label{def:cohframe}
	A frame is called {\it coherent} if every element of the frame can be expressed as a join of finite elements and the finite elements form a sublattice (equivalently, 1 is finite and the meet of two finite elements is finite).
\end{defn}
	{\it Spectral spaces}, introduced by Hochster in \cite{MR251026}, are topological spaces homeomorphic to the spectrum of a commutative ring. A {\it spectral map} between spectral spaces is a continuous map such that the inverse image of a quasi-compact open is quasi-compact. Every coherent frame corresponds uniquely to a spectral space. In fact, we have the following theorem:
\begin{Thm}[\cite{joyal:specsp}]\label{spec-cohfr}
	The category of spectral spaces and spectral maps is contravariantly equivalent to the category of coherent frames and coherent maps.
\end{Thm}
	For a spectral space $X$, Hochster \cite{MR251026} considered a new topology on $X$ by taking as basic open subsets the closed sets with quasi-compact complements. The space so obtained is called the {\it Hochster dual} of $X$ and it is denoted by $X^{\vee}$. He showed that the Hochster dual $X^\vee$ of any spectral space $X$ is also a spectral space and that ${X^\vee}^\vee = X$. Motivated by this, the Hochster dual of a coherent frame is defined as follows:
	\begin{defn}[\cite{kopi:hochdual}, Definition 1.2.4]
The Hochster dual of a coherent frame $F$ is its join completion. 
	\end{defn}
We recall that an ideal of a frame (in general for any lattice) is a down-set, closed under finite joins. An ideal $\mathcal{I}$ of a frame $F$ is called {\it prime} if $1\notin \mathcal{I}$ and if $a\wedge b\in \mathcal{I}$ implies either $a\in \mathcal{I}$ or $b\in \mathcal{I}$. The points of a frame $F$ correspond bijectively to prime ideals of $F$. Indeed, any point $x:F\longrightarrow \{0,1\}$ corresponds to the prime ideal $x^{-1}(0)$. Moreover, in any frame, every prime ideal $\mathcal{P}$ is principal and the generating element is $u_\mathcal{P}:=\underset{b\in \mathcal{P}}{\vee}b$. We have
$$\mathcal{P}=(u_\mathcal{P})=\{b\in F~|~b\leq u_\mathcal{P}\}$$
The generating element $u_\mathcal{P}$ of a prime ideal $\mathcal{P}$ is called a {\it prime element}. Therefore, we have the following natural bijections
$$\text{points}\leftrightarrow \text{prime ideals} \leftrightarrow \text{prime elements}$$
Let $(\mathcal{T}, \otimes, \mathbf{1})$ be a (commutative) tensor triangulated category.
We recall the definition of support on $(\mathcal{T}, \otimes, \mathbf{1})$
from \cite{kopi:hochdual}*{\S  3.2}:
\begin{defn} \label{def:kpsupport}
A support on $(\mathcal{T}, \otimes, \mathbf{1})$ is a pair $(F,d)$ where $F$ is a frame and $d:Ob(\mathcal{T})\longrightarrow F$ is a map satisfying
\begin{itemize}
	\item[(1)] $d(0)=0$ and $d(\mathbf{1})=1$,
	\item[(2)]$d(\sum a)= d(a), \quad\forall a \in Ob(\mathcal{T})$
	\item[(3)]$d(a\oplus b)=d(a)\vee d(b), \quad\forall a,b\in \mathcal{T}$
	\item[(4)] $d(a\otimes b)=d(a)\wedge d(b), \quad\forall a,b\in \mathcal{T}$
	\item[(5)] if $a\longrightarrow b\longrightarrow c \longrightarrow \sum a$ is a triangle in $\mathcal{T}$, then $d(b)\leq d(a)\vee d(c)$.
\end{itemize}
A morphism of supports from $(F,d)$ to $(F',d')$ is a frame map $F\longrightarrow F'$ compatible with the maps $d$ and $d'$.
\end{defn}

\section{Frames, Hochster duality and noncommutative tensor triangulated category}

{\bf Assumption : All primes of $\mathbf{K}$ are complete primes. }
%{\bf Assumption 2 : All finite elements of $\mathbf{K}$ are the principal radical thick tensor ideals i.e., of the form $\sqrt{k}$ for some $k\in \mathbf{K}$.}

One has a vast repertoire of examples where this assumption holds, and a
detailed description of the current knowledge about this can be found in the
introduction of \cite{nvy:supptpp}

\begin{defn}
	Let $S$ be a set of objects in a noncommutative tensor triangulated category $\mathbf{K}$. We define $G(S)$ to be the set of objects of $\mathbf{K}$ which are of the following forms:
	\begin{itemize}
		\item[(1)] an iterated suspension or desuspension of an object in $S$,
		\item[(2)] or a finite sum of objects in $S$,
		\item[(3)] or objects of the form $s\otimes t$ and $t\otimes s$ with $s\in S$ and $t\in \mathbf{K}$,
		\item[(4)] or an extension of two objects in $S$,
		\item[(5)] or a direct summand of an object in $S$.
	\end{itemize}
If $\mathbf{I}$ is a thick tensor ideal containing $S$, then clearly $G(S)\subseteq \mathbf{I}$. Hence, by induction, $G^\omega(S):= \bigcup_{n\in\mathbb{N}}G^n(S)\subseteq \mathbf{I}$. It may be easily verified that $G^\omega(S)$ is itself a thick tensor ideal and therefore it is the smallest thick tensor ideal containing $S$. We will denote it by $\langle S \rangle$.
\end{defn}
Recall that the radical of an ideal of a noncommutative ring is defined as the intersection of all the prime ideals containing it. In the same spirit, we give the following definition.
\begin{defn}
	We define the \emph{radical closure} of a thick tensor ideal $\mathbf{I}$ of a noncommutative tensor triangulated category $\mathbf{K}$ by
	\begin{equation*}
	\sqrt{\mathbf{I}}	:=\underset{\mathbf{I}\subseteq \mathbf{P}} {\bigcap}\mathbf{P}
	\end{equation*}
where $\mathbf{P}$ denotes the prime ideals of $\mathbf{K}$. If $\mathbf{I}$ is a thick tensor ideal such that $\mathbf{I}= \sqrt{\mathbf{I}}$, we call $\mathbf{I}$ radical.
\end{defn}

Clearly, any prime ideal is radical. It is also clear that if $\mathbf{I}$ is a thick tensor ideal, then $\sqrt{\mathbf{I}}$ is a radical thick tensor ideal. For any set of objects $S$, let $\sqrt{S}$ denote the radical
of the thick tensor ideal $\langle S \rangle$.

\begin{lem}\label{prod}
	Let $\mathbf{I}$ and $\mathbf{J}$ be two thick tensor ideals and let $S$ be a set of objects of $\mathbf{K}$. Then, $\{t\otimes s~|~t\in \mathbf{I}, s\in S\}\subseteq \mathbf{J}$ implies $\mathbf{I}\otimes \langle S \rangle \subseteq \mathbf{J}$.
\end{lem}
\begin{proof}
By definition, $\langle S \rangle = \bigcup_{n\in\mathbb{N}}G^n(S)$ and $G^0(S)= S$. So, by assumption, $\mathbf{I}\otimes G^0(S)\subseteq \mathbf{J}$. Suppose $\mathbf{I}\otimes G^m(S)\subseteq \mathbf{J}$ for some $0\neq m\in \mathbb{N}$. We will now show that $\mathbf{I}\otimes G^{m+1}(S)\subseteq \mathbf{J}$ i.e., 
\begin{equation}\label{subset}
	 \{t\otimes x~|~t\in \mathbf{I}\}\subseteq \mathbf{J}
\end{equation}
 for any $x\in  G^{m+1}(S)$. Obviously, \eqref{subset} is satisfied if $x$ is a finite sum of objects in $G^{m}(S)$. If $x$ is an iterated suspension or desuspension of an object in $G^{m}(S)$ or if $x$ is an extension of two objects in $G^{m}(S)$, then \eqref{subset} holds since $\otimes$ is biexact. If $x$ is a direct summand of an object in  $G^{m}(S)$, then \eqref{subset} holds since $\mathbf{J}$ is thick. If $x$ is of the form $s\otimes k$ or $k\otimes s$ for any $s\in G^{m}(S)$ and $k\in \mathbf{K}$, then clearly  \eqref{subset} holds since $\mathbf{I}$ and $\mathbf{J}$ are ideals. Thus, by induction, we obtain $\mathbf{I}\otimes G^n(S)\subseteq \mathbf{J}$ for all $n\in \mathbb{N}$. It follows that $\mathbf{I}\otimes \langle S \rangle \subseteq \mathbf{J}$.
\end{proof}

\begin{thm}\label{radical}
	Let $\mathbf{I}$ be a thick tensor ideal of $\mathbf{K}$. Then,  $\sqrt{\mathbf{I}}= \langle \{k\in \mathbf{K}~|~k^{\otimes n}\in \mathbf{I}\text{~for some~} n\in \mathbb{N}\} \rangle$.
	
\end{thm}
\begin{proof}
	Let $S:= \{k\in \mathbf{K}~|~k^{\otimes n}\in \mathbf{I}\text{~for some~} n\in \mathbb{N}\} $. Clearly, $S\subseteq \mathbf{P}$ for all prime ideals of $\mathbf{K}$ such that $\mathbf{P} \supseteq \mathbf{I}$. Hence, $\langle S\rangle \subseteq \sqrt{\mathbf{I}}$. 
	%Suppose, if possible, $\sqrt{S}\neq \sqrt{\mathbf{I}}$. Then, there exists some $t\in  \sqrt{\mathbf{I}}$ such that $S\cap \{t^{\otimes n}~|~n\in \mathbb{N}\}= \emptyset$
	Given $t\notin \langle S\rangle $, consider the collection $\Omega$ of all ideals $\mathbf{J}\supseteq \mathbf{I}$ such that $\mathbf{J}\cap \{t^{\otimes n}~|~n\in \mathbb{N}\}= \emptyset$. Clearly, $\mathbf{I}\in \Omega$ and the set $\Omega$ can be partially ordered by inclusion and any chain in $\Omega$ has an upper bound in $\Omega$. Therefore, by Zorn's Lemma there exists a maximal element, say $\mathbf{M}$, in $\Omega$. Thus, $\mathbf{M}\supseteq \mathbf{I}$ and $\mathbf{M}\cap \{t^{\otimes n}~|~n\in \mathbb{N}\}= \emptyset$. It is now enough to show that $\mathbf{M}$ is prime to prove that $t\notin \sqrt{\mathbf{I}}$. \\
	Let $k,k'\in \mathbf{K}$ be such that $k\otimes k'\in \mathbf{M}$. It may be easily verified using Lemma \ref{prod} that $ \langle \mathbf{M},k\rangle \otimes k'\subseteq \mathbf{M}$ and obviously we have $ \langle \mathbf{M},k\rangle \mathbf{M}\subseteq \mathbf{M}$. Therefore, again applying Lemma \ref{prod} we obtain $ \langle \mathbf{M},k\rangle \langle \mathbf{M},k'\rangle\subseteq \mathbf{M}$. Suppose, if possible, $k,k'\notin \mathbf{M}$. Since $\mathbf{M}$ is maximal, we must have $t^{\otimes n}\in \langle \mathbf{M},k\rangle $ and $t^{\otimes m}\in \langle \mathbf{M},k'\rangle$ for some $n,m \in \mathbb{N}$. This implies $t^{\otimes (n+m)}\in  \langle \mathbf{M},k\rangle \langle \mathbf{M},k'\rangle\subseteq \mathbf{M}$ which gives the required contradiction. 
\end{proof}
	
\begin{lem}\label{frame}
	Let ${\bf{Rad}}_{\mathbf{K}}$ denote the poset of radical ideals of a noncommutative tensor triangulated category $\mathbf{K}$ satisfying {\bf Assumption}. Then, ${\bf{Rad}}_{\mathbf{K}}$ is a frame with the following meet and join operations:
	\begin{equation*}
		\mathbf{I}_1\bigwedge\mathbf{I}_2 := \mathbf{I}_1\bigcap\mathbf{I}_2%\qquad  
	\end{equation*}
\begin{equation*}
	\bigvee_{j\in J}\mathbf{I}_j:= \sqrt{\bigcup_{j\in J}\mathbf{I}_j} \qquad 
\end{equation*}
for any two radical thick tensor ideals  $\mathbf{I}_1$ and $\mathbf{I}_2$ and for any set of radical thick tensor ideals $\{{{\mathbf{I}}_j}\}_{j\in J}$. 
\end{lem}
\begin{proof}
	By definition, $\bigvee_{j\in J}\mathbf{I}_j$ is a radical thick tensor ideal. Also, given a set of radical thick tensor ideals $\{{{\mathbf{I}}_i}\}_{i\in I}$, we have
	$$ \mathbf{I}_i= \sqrt{\mathbf{I}_i}=\bigcap_{\mathbf{I} _i\subseteq \mathbf{P}}\mathbf{P}\supseteq \underset{\bigcap_{i\in I}\mathbf{I}_i\subseteq \mathbf{P}}{\bigcap}\mathbf{P}=  \sqrt{\bigcap_{i\in I}\mathbf{I}_i}$$
	for every $i\in I$. Thus, we have $\bigcap_{i\in I}\mathbf{I}_i=\sqrt{\bigcap_{i\in I}\mathbf{I}_i}$. Hence, ${\bf{Rad}}_{\mathbf{K}}$ is a complete lattice. Let us now verify that $$\bigvee_{i\in I}(\mathbf{J}\bigcap \mathbf{I}_i)= \mathbf{J}\bigcap (\bigvee_{i\in I}\mathbf{I}_i)$$
	%Since $\mathbf{J}\bigcap \mathbf{I}_i\subseteq \mathbf{I}_i$ and $\mathbf{J}\bigcap \mathbf{I}_i\subseteq \mathbf{J}$ for all $i\in I$, 
	We clearly have $\bigvee_{i\in I}(\mathbf{J}\bigcap \mathbf{I}_i)\subseteq \mathbf{J}\bigcap (\bigvee_{i\in I}\mathbf{I}_i)$. Now, let $x\in \mathbf{J}\bigcap (\bigvee_{i\in I}\mathbf{I}_i)$. We define
	$$C_x:= \{k\in \bigvee_{i\in I} \mathbf{I}_i~|~x\otimes s\otimes  k,~ k\otimes s\otimes x\in \bigvee_{i\in I}(\mathbf{J}\bigcap \mathbf{I}_i)\text{~for all~} s\in \mathbf{K}\}$$
	We claim that $C_x=\bigvee_{i\in I}\mathbf{I}_i$. If our claim holds, then $x\in C_x$ which implies $x\otimes x\in \bigvee_{i\in I}(\mathbf{J}\bigcap \mathbf{I}_i)= \underset{\cup_{i\in I}\mathbf{J}\cap \mathbf{I}_i\subseteq \mathbf{P}}{\bigcap} \mathbf{P}$. Since, by assumption, all primes are complete, we have $x\in \bigvee_{i\in I}(\mathbf{J}\bigcap \mathbf{I}_i)$ and this completes the proof. We will now give a proof of our claim.\\
	First, we show that $C_x$ is a thick subcategory of $\mathbf{K}$. Let $k_1, k_2 \in \mathbf{K}$ be such that $k_1\oplus k_2\in C_x$. Then since $\bigvee_{i\in I}(\mathbf{J}\bigcap \mathbf{I}_i)$ is a thick subcategory, we have
	\begin{align*}
		&(k_1\oplus k_2) \otimes s \otimes x,~ x\otimes s \otimes (k_1\oplus k_2)\in \bigvee_{i\in I}(\mathbf{J}\bigcap \mathbf{I}_i)\\
		&\implies (k_1 \otimes s \otimes x) \oplus ( k_2  \otimes s\otimes x),~ (x \otimes s\otimes k_1)\oplus (x \otimes s\otimes k_2)\in  \bigvee_{i\in I}(\mathbf{J}\bigcap \mathbf{I}_i)\\
		&\implies k_1 \otimes s \otimes x, ~k_2 \otimes s \otimes x, ~x \otimes s\otimes k_1,~x\otimes s \otimes k_2 \in \bigvee_{i\in I}(\mathbf{J}\bigcap \mathbf{I}_i)\implies k_1,~ k_2\in C_x
	\end{align*} 
Next, we show that $C_x$ is a two-sided ideal. Let $k\in C_x$ and let $t\in \mathbf{K}$. Since $x\otimes s\otimes k,~ k\otimes s\otimes x\in  \bigvee_{i\in I}(\mathbf{J}\bigcap \mathbf{I}_i)$ for all $s\in \mathbf{K}$, we have  $x\otimes s\otimes (k\otimes t)=(x\otimes s\otimes k)\otimes t \in  \bigvee_{i\in I}(\mathbf{J}\bigcap \mathbf{I}_i)$ and $(k\otimes t)\otimes s\otimes x= k\otimes (t\otimes s)\otimes x \in  \bigvee_{i\in I}(\mathbf{J}\bigcap \mathbf{I}_i)$. This implies $k\otimes t\in C_x$. Similarly, $t\otimes k\in C_x$. \\
We will now show that $C_x$ is radical. Clearly, $C_x\subseteq \sqrt{C_x}$. Now, if possiblle, let $t\in  \sqrt{C_x}=\bigcap_{C_x\subseteq \mathbf{P}}\mathbf{P}$ be such that $t\notin C_x$. This implies either $x\otimes s\otimes t$ or $t\otimes s\otimes x$ does not belong to $\bigvee_{i\in I}(\mathbf{J}\bigcap \mathbf{I}_i)$ for some $s\in \mathbf{K}$. Without loss of generality, suppose $x\otimes s\otimes t\notin \bigvee_{i\in I}(\mathbf{J}\bigcap \mathbf{I}_i) = \underset{\cup_{i\in I}(\mathbf{J}\cap \mathbf{I}_i)\subseteq \mathbf{P}}{\bigcap} \mathbf{P}$. Then, there exists some prime ideal $\mathbf{P}_0\supseteq \bigcup_{i\in I}(\mathbf{J}\bigcap \mathbf{I}_i)$ such that $x\otimes s\otimes t\notin \mathbf{P}_0$. This implies $x\notin \mathbf{P}_0$ and  $t\notin \mathbf{P}_0$. For $k\in C_x$, we have 
\begin{align*}
x\otimes k,~ k\otimes x\in \bigvee_{i\in I}(\mathbf{J}\bigcap \mathbf{I}_i)\subseteq \mathbf{P}_0
\end{align*}
Hence $k\in \mathbf{P}_0$ showing that $C_x\subseteq \mathbf{P}_0$. Therefore, $t\in \sqrt{C_x}\subseteq \mathbf{P}_0$. This gives the required contradiction. Therefore, $C_x$ is radical. It is clear that $\mathbf{I}_i\subseteq C_x$ for all $i\in I$. Hence, $\bigcup_{i\in I} \mathbf{I}_i \subseteq C_x$ and since $C_x$ is radical, we must have $C_x= \bigvee_{i\in I}\mathbf{I}_i$.

 \end{proof}

\begin{Thm}\label{cohfr}
	The poset of radical thick tensor ideals ${\bf{Rad}}_{\mathbf{K}}$ of a noncommutative tensor triangulated category $\mathbf{K}$ satisfying {\bf Assumption} forms a coherent frame.
\end{Thm}

\begin{proof}
By Lemma \ref{frame}, we know that ${\bf{Rad}}_{\mathbf{K}}$ forms a frame. It is now enough to check that an element of the frame ${\bf{Rad}}_{\mathbf{K}}$ is finite if and only if  it is a principal radical thick tensor ideal i.e., of the form $\sqrt{k}$ for some $k\in \mathbf{K}$. Let $ \mathbf{I} $ be a finite element of the frame ${\bf{Rad}}_{\mathbf{K}}$. Then, clearly we have
\begin{equation*}
	\mathbf{I} \subseteq \sqrt{\bigcup_{k \in  \mathbf{I} } \sqrt{k}} = \bigvee_{k \in  \mathbf{I} }
	\sqrt{k}.
\end{equation*}
Since $ \mathbf{I} $ is radical, $k \in  \mathbf{I} $ implies $\sqrt{k} \subseteq  \mathbf{I} $. 
Thus, $ \mathbf{I}  = \bigvee_{k \in  \mathbf{I} } \sqrt{k}$.
Since $\mathbf{I}$ is finite, there exists $k_1,
\hdots, k_n \in  \mathbf{I} $ such that $ \mathbf{I}  \subseteq \sqrt{k_1} \vee \hdots \vee
\sqrt{k_n}$. We observe that $\mathbf{I} \subseteq \sqrt{\sqrt{k_1} \cup \dotsb \cup
	\sqrt{k_n}} \subseteq \sqrt{k_1 \oplus \dotsb \oplus k_n} \subseteq  \mathbf{I} $. 
Thus, $\mathbf{I}
= \sqrt{\sqrt{k_1} \cup \dotsb \cup \sqrt{k_n}} = \sqrt{k_1 \oplus \dotsb
	\oplus k_n}$. Therefore, $ \mathbf{I} $ is of the form $\sqrt{k}$.

Conversely, let $ \mathbf{I} = \sqrt{k_0}$ for some $k_0\in \mathbf{K}$. We need to check that $ \mathbf{I}$ is a finite
radical thick tensor ideal. Assume that
\begin{equation*}
	\sqrt{k_0} \subseteq \bigvee_{\lambda \in \Lambda} \mathbf{J}_{\lambda},
\end{equation*} 
where the $\mathbf{J}_{\lambda}$ are radical thick tensor ideals. Then in particular $k_0 \in
\sqrt{\bigcup_{\lambda \in \Lambda} \mathbf{J}_{\lambda}}$. Let us denote
$\bigcup_{\lambda \in \Lambda} \mathbf{J}_{\lambda}$ by $S$. Thus, by Proposition \ref{radical}, 
\begin{equation*}
	k_0 \in \sqrt{\langle S\rangle} = \sqrt{\langle{\bigcup_{\lambda \in \Lambda}
			\mathbf{J}_{\lambda}}\rangle}= \langle \{k\in \mathbf{K}~|~k^{\otimes n}\in \langle{\bigcup_{\lambda \in \Lambda}
		\mathbf{J}_{\lambda}}\rangle\text{~for some~} n\in \mathbb{N}\} \rangle
\end{equation*}
Let the finitely
many elements of $ \{k\in \mathbf{K}~|~k^{\otimes n}\in \langle{\bigcup_{\lambda \in \Lambda}
	\mathbf{J}_{\lambda}}\rangle \text{~for some~} n\in \mathbb{N} \}$ involved in the iterative construction of $k_0$ be $k_1, \hdots, k_r$ i.e.,
\begin{equation}\label{fg}
	k_0\in\langle k_1,\hdots,k_r\rangle \subseteq \sqrt{\langle S\rangle} 
\end{equation}

Then, we have $k_i^{\otimes n_i}\in  \langle{\bigcup_{\lambda \in \Lambda}
	\mathbf{J}_{\lambda}}\rangle$ for some $n_i\in \mathbb{N}$ for each $i=1,\hdots, r$.
Let the finitely
many elements of $\bigcup_{\lambda\in\Lambda} \mathbf{J}_{\lambda}$ involved in the iterative construction of $k_i^{\otimes n_i}, i=1,\hdots,r$ be $x_1,\hdots, x_m$. Suppose $\{x_1, \hdots, x_m\} \subseteq
\mathbf{J}_{\lambda_1} \cup \hdots \cup \mathbf{J}_{\lambda_{\nu}}$ for some $\nu \in \mathbb{N}$.
Thus, for each $ i=1,\hdots,r$, we have\begin{equation*}
	k_i^{\otimes n_i} \in \langle \{x_1,\hdots,x_m\}\rangle \subseteq
	\langle{\mathbf{J}_{\lambda_1} \cup \dotsb \cup \mathbf{J}_{\lambda_r}}\rangle \implies k_i \in \sqrt{\langle \mathbf{J}_{\lambda_1} \cup \dotsb \cup \mathbf{J}_{\lambda_r} \rangle} =
	\mathbf{J}_{\lambda_1} \vee \dotsb \vee \mathbf{J}_{\lambda_r}
\end{equation*}

and hence by \eqref{fg}
\begin{equation*}
	k_0\in\langle k_1,\hdots,k_r\rangle \subseteq \mathbf{J}_{\lambda_1} \vee \dotsb \vee \mathbf{J}_{\lambda_r}
\end{equation*}
Thus, $\sqrt{k_0}\subseteq  \mathbf{J}_{\lambda_1} \vee \dotsb \vee \mathbf{J}_{\lambda_r}$ which proves that $\sqrt{k_0}$ is finite.
\end{proof}

	\begin{defn}\label{supp}
		We call the frame of radical thick tensor ideals of a noncommutative tensor triangulated category $\mathbf{K}$ satisfying {\bf Assumption} the Zariski frame of $\mathbf{K}$ and we denote it by $  {\bf Zar(\mathbf{K})}$. By the Zariski spectrum of $\mathbf{K}$ we mean the spectral space associated to ${\bf Zar(\mathbf{K})}$ and we denote it by $Spec_{Zar}(\mathbf{K})$.
	\end{defn}
Next, we introduce a notion of support for a noncommutative tensor triangulated category.

\begin{defn} \label{def:noncommfts}
	A support on $\mathbf{K}$ is a pair $(F,d) $ where $F$ is a frame and $d: Ob(\mathbf{K})\longrightarrow F$ is a map satisfying:
	\begin{itemize}
		\item[(1)] $d(0)=0$ and $d(\mathbf{1})=1$
		\item[(2)] $d(\sum k) = d(k)~~\forall k\in \mathbf{K}$
		\item[(3)] $d(k\oplus t) = d(k)\vee d(t)~~ \forall k,t \in \mathbf{K}$
		\item[(4)] $d(k\otimes t) = d(k)\wedge d(t)= d(t\otimes k)~~ \forall k,t \in \mathbf{K}$
		\item[(5)] If $k\longrightarrow t\longrightarrow r\longrightarrow \sum k$ is a triangle in $\mathbf{K}$, then $d(t)\leq d(k) \vee d(r)$
	\end{itemize}
A
morphism $\varphi \colon (F, d) \longrightarrow (F', d')$ is a morphism of
frames $\varphi \colon F \longrightarrow F'$ such that $d'(a) =
\varphi(d(a))$ for all $a \in Ob(\mathbf{K})$.
\end{defn}

\begin{rem}
  This definition of support is motivated by the one in \cite{kopi:hochdual}, recalled
  here in definition \ref{def:kpsupport}. We shall see that under \textbf{Assumption} this
  continues to correspond to a notion of support on
  noncommutative tensor triangulated categories as described by Nakano, Vashaw and
  Yakimov.
\end{rem}

% Commented out by Vivek
% \begin{rem}
%   This notion of support does not exactly match with Nakano-Vashaw-Yakimov's
%   definitions at a first glance. However, in the case where all primes are
%   complete, this match using theorem 3.2.1 and theorem 2.3.2(a) in
%   \cite{nvy:supptpp}.
% \end{rem}

%Let $2^{\bf{K}}$ denote the set of all subsets of objects of $\bf{K}$. 
%\begin{defn}
%	A support $(F,d)$ on $\bf{K}$ is called geometric if $$\bar{d}: 2^{\bf{K}}\longrightarrow F,\quad \bar{d}(S):=\bigvee_{k\in S}d(k) $$
%	satisfies $\bar{d}(S) = \bar{d}(\sqrt{S})$ for any $S\in 2^{\bf{K}}$.
%\end{defn}

%\begin{lem}\label{barframe}
%Let $(F,d)$ on $\bf{K}$  be a geometric support. Then, $\bar{d}\big|_{Zar(\bf{K})}: Zar(\bf{K})\longrightarrow F$ is a morphism of frames.
%\end{lem}

\begin{Thm} \label{thm:initsupp}
	Let $\mathbf{K}$ be a noncommutative tensor triangulated category satisfying {\bf Assumption}. Then the assignment $s: Ob(\mathbf{K})\longrightarrow  {\bf Zar(\mathbf{K})}, ~~k\mapsto \sqrt{k}$ is a support. Moreover, it is initial among all supports.
\end{Thm}
\begin{proof}
We have  $\sqrt{\mathbf{1}}=\underset{\mathbf{1}\in \mathbf{P}}{\bigcap} \mathbf{P}=\underset{\emptyset}{\bigcap} = \langle\mathbf{1}\rangle$ and $\mathbf{I}\vee \sqrt{0}= {\underset{\mathbf{I}=\mathbf{I}\cup \sqrt{0}~\subseteq~\mathbf{P}}{\bigcap} \mathbf{P}}=\mathbf{I}$ for any $\mathbf{I}\in 	\bf Zar(\mathbf{K})$. Thus, clearly condition $(1)$ in Definition \ref{supp} is satisfied. Also, conditions $(2)$ and $(5)$ are clearly satisfied using the fact that prime ideals are triangulated subcategories of $\mathbf{K}$. Also, since prime ideals are thick, we have $\sqrt{k}\subseteq \sqrt{k\oplus t}$ and $\sqrt{t}\subseteq \sqrt{k\oplus t}$ for any $ k,t \in \mathbf{K}$ and so $\sqrt{k}\vee \sqrt{t}\subseteq \sqrt{k\oplus t}$. Conversely, we clearly have $k\oplus t\in \sqrt{k}\vee \sqrt{t}$. Hence, condition $(3)$ is also satisfied. Let us now check condition $(4)$.  It is clear that $\sqrt{k\otimes t}\subseteq \sqrt{k}$ and $\sqrt{k\otimes t}\subseteq \sqrt{t}$ and therefore $\sqrt{k\otimes t}\subseteq \sqrt{k}\cap \sqrt{t}$. Finally, we have $\sqrt{k}\cap \sqrt{t}\subseteq \sqrt{k\otimes t}$ since all primes are complete, by \textbf{Assumption}. This shows that $s$ is a support.\\
We will now show that  $s$ is initial among all supports. Let $d:Ob(\mathbf{K})\longrightarrow  F$ be an arbitrary support. Since ${\bf Zar(\mathbf{K})}$ is coherent, every element is a join of finite elements and so any frame map ${\bf Zar(\mathbf{K})}\longrightarrow F$ is completely determined by its value on finite elements. So consider the frame map $u: {\bf Zar(\mathbf{K})}\longrightarrow F$ given by $\sqrt{k}\mapsto d(k)$.  Clearly, we have $u\circ s= d$. In fact, it is also clear that  there cannot be any other choice of map ${\bf Zar(\mathbf{K})}\longrightarrow F$ which is compatible with $s$ and $d$. So there is at most one support map $u$. Let us now check that $u$ is well defined.  Let $k,t \in \mathbf{K}$ be such that  $\sqrt{k}=\sqrt{t}$. We define $I(t)=\{r\in \mathbf{K}~|~d(r)\leq d(t)\}$. It follows from the properties of the support $d$ that $I(t)$ is a thick tensor ideal. Moreover, if $s\in \mathbf{K}$ be such that $s^{\otimes n}\in I(t)$ for some $n\in \mathbb{N}$, then $s\in I(t)$ since $d(s^{\otimes n})= d(s)$. Therefore, it follows from Proposition \ref{radical} that $I(t)$ is a radical thick tensor ideal containing $t$ and hence $\sqrt{t}$. Since $\sqrt{k}\subset \sqrt{t}\subseteq I(t)$ we have $d(k)\leq d(t)$ and by symmetry we obtain our desired result.

\end{proof}

We now proceed to show that if $\mathbf{K}$ is a noncommutative tensor triangulated category satisfying {\bf Assumption}, then the noncommutative Balmer spectrum $Spc(\mathbf{K})$ is the Hochster dual of the Zariski spectrum $Spec_{Zar}(\mathbf{K})$ of $\mathbf{K}$.

\begin{Thm}\label{corres}
Let $\mathbf{K}$ be a noncommutative tensor triangulated category satisfying {\bf Assumption}. Then,
\begin{itemize}
	\item[(1)] the frame-theoretic points of $\bf{Zar}(\mathbf{K})$ correspond bijectively to prime thick tensor ideals in $\mathbf{K}$.
	\item[(2)] Under the above correspondence, a finite element $\sqrt{k}$ of $\bf{Zar}(\mathbf{K})$ corresponds to the set of prime thick tensor ideals $\{\mathbf{P}\in Spc(\mathbf{K})~|~k\notin \mathbf{P}\}$.
\end{itemize}
	\end{Thm}

\begin{proof}
	\begin{itemize}
		\item[(1)] Recall that for any frame $F$, the frame-theoretic points of $F$ correspond bijectively to the prime ideals of $F$ and the prime ideals in turn are in natural bijection with the prime elements of $F$. Now, we put $F= \bf{Zar}(\mathbf{K})$. For any point $x:{\bf{Zar}(\mathbf{K})}\longrightarrow \{0,1\}$ of $\bf{Zar}(\mathbf{K})$, the corresponding prime ideal of $\bf{Zar}(\mathbf{K})$ is given by $\mathfrak{ p}_x:= x^{-1}(0)$ and the corresponding prime element of $\bf{Zar}(\mathbf{K})$ is given by $\mathbf{I}_x:= \bigvee_{\mathbf{I}\in  \mathfrak{ p}_x}\mathbf{I}$. We also know $\mathfrak{p}_x=(\mathbf{I})_x=\{\mathbf{I}\in \mathbf{Zar}(\mathbf{K})~|~\mathbf{I}\subseteq {\mathbf{I}}_x\}$. We will now show that $\mathbf{I}_x$ is prime. Let $\mathbf{J}_1$ and $\mathbf{J}_2$ be thick tensor ideals such that $\mathbf{J}_1\otimes \mathbf{J}_2\subseteq \mathbf{I}_x$. Clearly, $(\mathbf{J}_1\cap\mathbf{J}_2)^{\otimes 2}\subseteq \mathbf{J}_1\otimes \mathbf{J}_2\subseteq \mathbf{I}_x$. Since $\mathbf{I}_x$ is radical, we have $\mathbf{J}_1\cap\mathbf{J}_2\subseteq \mathbf{I}_x$ by Proposition \ref{radical}. Therefore, we have $ \mathbf{J}_1\wedge\mathbf{J}_2= \mathbf{J}_1\cap\mathbf{J}_2\in (\mathbf{I}_x)=\mathfrak{ p}_x$ in the frame $\bf{Zar}(\mathbf{K})$. Since $\mathfrak{p}_x$ is a prime ideal of $\bf{Zar}(\mathbf{K})$ we must have $\mathbf{J}_1\in \mathfrak{p}_x$ or $\mathbf{J}_2\in \mathfrak{p}_x$. In other words, we have $\mathbf{J}_1\subseteq \mathbf{I}_x$ or $\mathbf{J}_2\subseteq \mathbf{I}_x$ which proves that $\mathbf{I}_x$ is prime.
		Thus, we obtain the following well defined map of sets
		\begin{equation}\label{bijec}
			\{\text{frame-theoretic points of $\bf{Zar}(\mathbf{K})$}\}\longrightarrow \{\text{prime thick tensor ideals in $\mathbf{K}$}\}\qquad\qquad x\mapsto \mathbf{I}_x
		\end{equation}
		If $\mathbf{I}_x= \mathbf{I}_y$, then we have $x^{-1}(0)=\mathfrak{p}_x=(\mathbf{I}_x)=(\mathbf{I}_y)=\mathfrak{p}_y=y^{-1}(0)$. Clearly, this implies $x=y$ showing that \eqref{bijec} is an injection. To show surjection, let $\mathbf{P}$ be any prime thick tensor ideal of $\mathbf{K}$. It may be easily verified that $(\mathbf{P}):= \{\mathbf{I}\in \bf{Zar}(\mathbf{K})~|~\mathbf{I}\subseteq \mathbf{P}\}$ defines a prime ideal of the frame $\bf{Zar}(\mathbf{K})$. Now, we define $y: \mathbf{Zar({K})}\longrightarrow \{0,1\}$ by
		\[
		y(\mathbf{I}):=
		\begin{cases}
			0  & \text{ if $\mathbf{I}\in (\mathbf{P})$} \\
			1 & \text{otherwise}
		\end{cases}
		\]
		It may be easily verified that $y: \mathbf{Zar}(\mathbf{K})\longrightarrow \{0,1\}$ is a morphism of frames which shows that $y$ is a frame-theoretic point of $\bf{Zar}(\mathbf{K})$. Since $\mathfrak{p}_y:=y^{-1}(0) =(\mathbf{P})$, it follows that $y$ maps to $\mathbf{P}$ under \eqref{bijec}.
		
		\item[(2)] The open set corresponding to the finite element $\sqrt{k}$ of the coherent frame $\mathbf{Zar({K})}$ is $\{x:\mathbf{Zar}(\mathbf{K})\longrightarrow \{0,1\}~|~x(\sqrt{k})=1\}$. Clearly, we have 
		$$x(\sqrt{k})=1 \iff \sqrt{k}\notin \mathfrak{p}_x= x^{-1}(0)= (\mathbf{I}_x)\iff k\notin \mathbf{I}_x$$
		Since \eqref{bijec} is a bijection, it follows that the set $\{x:\mathbf{Zar}(\mathbf{K})\longrightarrow \{0,1\}~|~x(\sqrt{k})=1\}$ corresponds bijectively to the set of prime thick tensor ideals $\{\mathbf{P}\in Spc(\mathbf{K})~|~k\notin \mathbf{P}\}$.
	\end{itemize}
	
\end{proof}

\begin{cor}\label{Hdual}
	Let $\mathbf{K}$ be a noncommutative tensor triangulated category satisfying {\bf Assumption}. The noncommutative Balmer's spectrum $Spc(\mathbf{K})$ of $\mathbf{K}$ is the Hochster dual of the Zariski
	spectrum $Spec_{Zar}(\mathbf{K})$. 
\end{cor}
\begin{proof}
	The topology of $Spc(\mathbf{K})$ is given by open sets which are complements of the sets of the form $\{\mathbf{P}\in Spc(\mathbf{K})~|~k\notin \mathbf{P}\}$. The result therefore  follows from Theorem  \ref{corres}.
\end{proof}

From the frame theoretic support data, one can reconstruct the support data $V
\colon \mathbf{K} \longrightarrow \mathcal{X}_{cl}(Spc(\mathbf{K}))$
described by Nakano, Vashaw and Yakimov \cite{nvy:supptpp}*{Definition 2.3.1}. This is described below in terms of a functorial equivalence between the frame theoretic support data and the support data taking values in closed subsets of $Spc(\mathbf{K})$.

\begin{cnst} \label{cns:suppdatalinks}
We briefly recall the construction of a topological support data
corresponding to a frame theoretic support data. Suppose that $d \colon
\mathbf{K} \to F$ is a frame-theoretic support data and that $F$ is
coherent. Let $X_F$ be the spectral space corresponding to $F$ (see Theorem
\ref{spec-cohfr}). We know that the points in $X_F$
correspond to frame maps $p \colon F \to \{0, 1\}$ and the topology consists of
open sets $U_f = \{p \in X_F\,|\, p(f) = 1\}$. Let $Y_F$ be the Hochster
dual of $X_F$, where the open sets are closed subsets of $X_F$ with
quasi-compact complement.  Consider the assignment $\sigma \colon \mathbf{K}
\longrightarrow \mathcal{X}_{cl}(Y_F)$ given by
\begin{equation*}
  \sigma(a) = \{p \in X_F \,|\, p(d(a)) = 1 \}.
\end{equation*}
This is well defined (see remark \ref{rem:suppclosed}).

One also has a reverse construction.  By the Lemma in \cite{johnstone:stonespaces}*{page 41},
the closed subsets of $Y_F$ are in one-to-one correspondence with elements
of $F$. Thus, given a support $\sigma \colon \mathbf{K} \longrightarrow
\mathcal{X}_{cl}(Y_F)$, one can define for $a \in \mathbf{K}$, $d(a)$ to be
the element of $F$ corresponding to the closed subset $\sigma(a)$. When
$\sigma$ satisfies the tensor product property, $d$ turns out to be a frame
theoretic support in the sense of definition \ref{def:noncommfts}.
\end{cnst}

\begin{rem} \label{rem:suppclosed}
In the argument below we shall need the fact that the support $\sigma(a)$ constructed above are closed subsets of $Y_F$. This is standard and can be seen as follows. The space $Y_F$ is denoted by $(Spec\,F)_{\text{inv}}$ in \cite{MR3929704}. The subsets of the form $\set{p \colon F \longrightarrow \{0,1\}}{p(a) = 1}$ are quasi-compact by \cite{MR3929704}*{2.2.3(c)} as $Spec\,F$ inherits the subspace topology from $2^F$.
\end{rem}

\begin{notn} \label{not:suppdata}
  Let $\mathcal{F}$ be the category of support data for $\mathbf{K}$ taking
  values in coherent frames as in definition \ref{def:noncommfts}. 

  Let $\mathcal{S}$ be the category of support data for $\mathbf{K}$ taking
  values in spectral topological spaces, along with the restriction that the
  support data should have the tensor product property. In other words, an
  object $(X, \sigma)$ in $\mathcal{S}$ is a support data $\sigma \colon
  \mathbf{K} \longrightarrow \mathcal{X}_{cl}(X)$, where $X$ is a spectral
  topological space and $\sigma$ in additon to being a support data as
  defined in Definition \ref{def:nvysuppdata}, satisfies $\sigma(a \otimes b) =
  \sigma(a) \cap \sigma(b)$. A morphism $\psi \colon (X, \sigma)
  \longrightarrow (X', \sigma')$ in $\mathcal{S}$ is a continuous map $f \colon X
  \longrightarrow X'$ such that for all $a \in Ob(\mathbf{K})$,
  $\sigma(a) = f^{-1}(\sigma'(a))$.
\end{notn}

\begin{lem} \label{lem:ftsptpsp}
  Suppose $(F, d)$ belongs to $\mathcal{F}$. Then the support data $(Y_F,
  \sigma)$ constructed in \ref{cns:suppdatalinks} belongs to $\mathcal{S}$.
  Conversely, if $(X, \tau)$ belongs to $\mathcal{S}$, then the reverse
  construction gives an element $(F_X, d)$ of $\mathcal{F}$.
\end{lem}

\begin{proof}
  Suppose $(F, d) \in \mathcal{F}$.
  That $d(0) = 0$ and $d(1) = 1$ (property (1) in definition
  \ref{def:noncommfts}) implies that $\sigma(0) = \emptyset$ and
  $\sigma(\mathbf{1}) = X$ and that $d(\Sigma k) = d(k)$ for all $k \in
  Ob(\mathbf{K})$ implies that $\sigma(\Sigma(k)) = \sigma(k)$ for all
  $k$ is straightforward.

  We know that $d(k \oplus t) = d(k) \vee d(t)$ for all $k, t \in
  Ob(\mathbf{K})$. Thus,
  \begin{align*}
    \sigma(k \oplus t)
    &= \set{p}{p(d(k \oplus t)) = 1}
    = \set{p}{p(d(k) \vee d(t)) = 1}
    = \set{p}{p(d(k)) \vee p(d(t)) = 1} \\
    &= \set{p}{p(d(k)) = 1 \text{ or } p(d(t)) = 1}
    = \set{p}{p(d(k)) = 1} \cup \set{p}{p(d(t)) = 1} \\
    &= \sigma(k) \cup \sigma(t).
  \end{align*}
  proving (2) of definition \ref{def:nvysuppdata}.

  Property (5) in definition \ref{def:noncommfts} implies that if $a \to b
  \to c \to \Sigma a$ is a distinguished triangle then $d(a) \leq
  d(b) \vee d(c)$. Thus, for any point $p \colon F \to \{0, 1\}$,
  $p(d(a)) \leq p(d(b)) \vee p(d(c))$. We have
  \begin{align*}
    \sigma(a)
    &= \set{p}{p(d(a)) = 1} \\
    &\subseteq \set{p}{p(d(b)) \vee p(d(c)) = 1} 
    = \set{p}{p(d(b)) = 1 \text{ or } p(d(c)) = 1} \\
    &= \sigma(b) \cup \sigma(c).
  \end{align*}
  This establishes \ref{def:nvysuppdata}(4). It remains to check (5). Note
  that for any $a, b \in Ob(\mathbf{K})$,
  \begin{align*}
    \sigma(a \otimes b)
    &= \set{p}{p(d(a \otimes b)) = 1}
    = \set{p}{p(d(a)) \wedge p(d(b)) = 1} \quad \text{(by definition
    \ref{def:noncommfts}(4))} \\
    &= \set{p}{p(d(a)) = 1} \cap \set{p}{p(d(b)) = 1}
    = \sigma(a) \cap \sigma(b)
  \end{align*}
  Thus such a support satisfies the tensor product property and hence also
  satisfies the property (5). This completes the proof that $(Y, \sigma)$
  constructed in \ref{cns:suppdatalinks} belongs to $\mathcal{S}$.

  Conversely suppose $(X, \sigma) \in \mathcal{S}$. Let us denote the
  coherent frame corresponding to $X$ by $F_X$, and $d(a)$ to be the element of
  $F_X$ corresponding to the closed subset $\sigma(a) = \set{p \in Y_{X_F}}{p(d(a))
  = 1}$. Now observe the equalities:
  \begin{align*}
    \sigma(a) \cup \sigma(b)
    &= \set{p}{p(d(a)) = 1} \cup \set{p}{p(d(b)) = 1}
    = \set{p}{p(d(a)) = 1 \text{ or } p(d(b)) = 1}
    = \set{p}{p(d(a) \vee d(b)) = 1}. \\
    \sigma(a) \cap \sigma(b)
    &= \set{p}{p(d(a)) = 1} \cap \set{p}{p(d(b)) = 1}
    = \set{p}{p(d(a)) = 1 \text{ and } p(d(b)) = 1}
    = \set{p}{p(d(a) \wedge d(b)) = 1}. 
  \end{align*}
  These equalities, along with the fact that the subset $\set{p}{p(f) = 1}$
  uniquely determines $f$, and the computations above give us the fact that
  $d$ satisfies all the properties listed in definition
  \ref{def:noncommfts}. Thus $(F_X, d) \in \mathcal{F}$.
\end{proof}

\begin{defn}
  \label{def:eqsigmaF}
  Let $\Xi \colon \mathcal{F} \longrightarrow \mathcal{S}$ and $\Gamma \colon
  \mathcal{S} \longrightarrow \mathcal{F}$ be the two maps constructed in
  \ref{cns:suppdatalinks}. These are well defined by lemma
  \ref{lem:ftsptpsp}.
\end{defn}

\begin{lem}
  $\Xi$ and $\Gamma$ are contravariant functors inducing equivalences between
  $\mathcal{F}$ and $\mathcal{S}$.
\end{lem}
\begin{proof}
  That $\Sigma \circ \Xi$ and $\Xi \circ \Sigma$ are naturally isomorphic to
  the identity functors follow from the correspondence between spectral
  spaces and coherent frames and the computations done in the proof of
 Lemma \ref{lem:ftsptpsp}.
\end{proof}

\begin{lem}
  $(F_0, d_0)$ is an initial support datum in $\mathcal{F}$ if and only if the
  corresponding support datum $(Y_{F_0}, \sigma_0) = \Xi( (F_0, d_0) )$ is a
  final support datum.
\end{lem}
\begin{proof}
  This follows from the contravariance of the functors involved.
\end{proof}

\begin{thm} \label{pro:cmpkpnvy}
  Let $\mathbf{K}$ be a noncommutative tensor triangulated category
  satisfying \textbf{Assumption}, and $Spc(\mathbf{K})$ be the corresponding noncommutative
  Balmer spectrum. The support data given by $k \mapsto \sqrt{k}$ taking
  values in $\bf{Zar}(\mathbf{K})$ induces a support (in the sense of
  definition \ref{def:nvysuppdata}) on $Spc(\mathbf{K})$. Moreover, this
  support data matches with the one given by $V \colon \mathbf{K}
  \longrightarrow \mathcal{X}_{cl}(Spc (\mathbf{K}))$:
  \begin{equation*}
    V(A) = \{ \mathbf{P} \in Spc(\mathbf{K}) \,|\, A \notin \mathbf{P} \}.
  \end{equation*}
\end{thm}
\begin{proof}
  We know that the support datum $({\bf Zar}(\mathbf{K}), \sqrt{\ })$ is
  initial in $\mathcal{F}$. Thus $\Xi\bigl(({\bf Zar}(\mathbf{K}), \sqrt{\
  })\bigr)$ is a final support datum in $\mathcal{S}$, i.e.\ it is final
  among all support data of the form $(X, \sigma \colon \mathbf{K}
  \longrightarrow Spc(\mathbf{K}))$. But by Theorem 2.3.2(a) in
  \cite{nvy:supptpp}, $V$ described above is also the final support datum in
  $\mathcal{S}$. Thus by universality of final objects, there is a natural
  isomorphism between $\Xi\bigl(({\bf Zar}(\mathbf{K}), \sqrt{\ })\bigr)$
  and $(Spc(\mathbf{K}, V)$.
\end{proof}

\begin{rem}
  Proposition \ref{pro:cmpkpnvy} shows that the frame-theoretic methods
  reconstructs the support function $V$ from a categorical viewpoint.
\end{rem}

Next, we show that the bijective correspondence between the radical thick tensor ideals and the open subsets of $Spc(\mathbf{K})^\vee$ can be promoted to a homeomorphism of spectral spaces.
\begin{Thm}\label{noncomTN}
Let $\mathbf{K}$ be a noncommutative tensor triangulated category satisfying {\bf Assumption} and let $Spc(\mathbf{K})^\vee$ be the Hochster dual of the noncommutative Balmer's spectrum $Spc(\mathbf{K})$. Then, the following spaces are spectral and there is a homeomorphism between them:
\begin{itemize}
\item[(1)] The frame $\bf{Zar}(\mathbf{K})$ of radical thick tensor ideals of $\mathbf{K}$ endowed with the topology generated by the open sets
\begin{equation}\label{lit1}
\{\mathbf{I}\in \bf{Zar}(\mathbf{K})~|~k\notin \mathbf{I}\}\quad \quad \forall~k\in \mathbf{K}
\end{equation}
\item[(2)] The poset $\Omega(Spc(\mathbf{K})^\vee)$ of open subsets of $Spc(\mathbf{K})^\vee$ (or equivalently, open subsets of $Spec_{Zar}(\mathbf{K})$) endowed with the topology generated by the open sets
\begin{equation}\label{lit2}
\{V\in \Omega(Spc(\mathbf{K})^\vee)~|~ V\nsupseteq U\}\quad \quad \forall~U\in \Omega(Spc(\mathbf{K})^\vee)
\end{equation}
\end{itemize}
\end{Thm}
\begin{proof}
By Theorem \ref{cohfr}, we know that the radical thick tensor ideals $\bf{Zar}(\mathbf{K})$ of a a noncommutative tensor triangulated category $\mathbf{K}$ forms a coherent frame. Thus, by \cite[Proposition 4.1]{MR3787524}, the set $\bf{Zar}(\mathbf{K})$ endowed with the lower interval topology is a spectral space. The lower interval topology on $\bf{Zar}(\mathbf{K})$ is generated by the open sets ($\forall~\mathbf{I}\in  \bf{Zar}(\mathbf{K})$)
$$L(\mathbf{I})=\{\mathbf{J}\in  \bf{Zar}(\mathbf{K})~|~\mathbf{J}\nsupseteq \mathbf{I}\}=\bigcup_{k\in \mathbf{I}}\{\mathbf{J}\in  \bf{Zar}(\mathbf{K})~|~k\notin \mathbf{J}\}$$
In other words, the lower interval topology on $\bf{Zar}(\mathbf{K})$ is generated by the collection in \eqref{lit1}. By Theorem \ref{spec-cohfr} and Corollary \ref{Hdual}, we know that there is an order-preserving bijective correspondence between the radical thick tensor ideals and the open subsets of $Spc(\mathbf{K})^\vee$. Clearly, the lower interval topology on the frame of open subsets of  $Spc(\mathbf{K})^\vee$  is generated by the collection in \eqref{lit2}. Thus, we have the required homeomorphism.
\end{proof}
\begin{rem}
Hilbert’s Nullstellensatz is the most fundamental theorem in algebraic geometry which establishes a bridge between geometry and algebra by relating algebraic sets to ideals in polynomial rings over algebraically closed fields. Another classical fact is that the closed subspaces of the spectrum $Spec(R)$ of a commutative ring $R$ are in bijective correspondence with radical ideals of $R$ which can be viewed as a nullstellensatz-like result. A topological enhancement of this nullstellensatz-like result was provided by Finocchiaro, Fontana and Spirito in \cite{MR3513063} where they showed that this bijective correspondence can be promoted to a homeomorphism. In \cite{MR3787524}, Banerjee provided a similar ``topological nullstellensatz"-like result for a (commutative) tensor triangulated category. Our Theorem \ref{pro:cmpkpnvy} could be seen as a ``topological nullstellensatz" for a noncommutative tensor triangulated category.
\end{rem}

	\bibliographystyle{plain}
	\bibliography{references}
	
\end{document}